\documentclass[11pt]{amsart}

\usepackage{geometry}
\usepackage{amssymb}
\usepackage{mathtools}
\usepackage{xcolor}
\usepackage[colorlinks=true,citecolor=red,linkcolor=blue,urlcolor=red]{hyperref}

\DeclareMathOperator{\rd}{rd}

\newtheorem{thm}{Theorem}[section]
\newtheorem{lem}[thm]{Lemma}
\newtheorem{prop}[thm]{Proposition}

\newtheorem{conj}[thm]{Conjecture}

\theoremstyle{definition}

\newtheorem{rem}[thm]{Remark}

\numberwithin{equation}{section}

\begin{document}

\title{A counterexample to the near-quadratic Elekes--R\'onyai expander conjecture over \texorpdfstring{$\mathbb R$}{R}}

\author{Jihao Liu}
\address{Department of Mathematics, Peking University, No. 5 Yiheyuan Road, Haidian District, Beijing 100871, China}
\address{Beijing International Center for Mathematical Research, Peking University, No. 5 Yiheyuan Road, Haidian District, Beijing 100871, China}
\email{liujihao@math.pku.edu.cn}

\date{}

\begin{abstract}
We disprove the near-quadratic Elekes--R\'onyai expander conjecture over $\mathbb R$. The counterexample is a fixed nonspecial quadratic polynomial, together with arbitrarily large finite sets of real algebraic integers on which its image has a fixed power saving from quadratic size. The main result of this paper is obtained by generative AI, particularly ChatGPT 5.5 Pro and the Rethlas system. The proof relies on a recent construction by OpenAI of an infinite tower of number fields.
\end{abstract}

\maketitle

\section{Introduction}\label{sec:introduction}

Elekes and R\'onyai proved that a real polynomial in two variables expands Cartesian products unless it has an additive or multiplicative form \cite[Theorem 2]{ER00}.
More precisely, the exceptional forms are
\[
f(x,y)=h(u(x)+v(y))
\quad\text{or}\quad
f(x,y)=h(u(x)v(y)),
\]
where $h,u,v\in\mathbb R[t]$ are univariate polynomials.
Quantitative versions, including the exponent $4/3$ for bounded degree polynomials, were proved by Raz--Sharir--Solymosi \cite{RSS16}; see the survey \cite[Paragraph after Theorem 1.1]{deZ18} for more details on the history on this question. Motivated by the gap between this exponent and the quadratic upper bound, Makhul--Roche-Newton--Warren--de Zeeuw asked for a near-quadratic form of the Elekes--R\'onyai phenomenon and gave a subgraph construction showing that related incidence statements are delicate \cite[Paragraph after Theorem 1.5]{MRNWdZ20}. The near-quadratic form may be stated as follows.

\begin{conj}[Near-quadratic Elekes--R\'onyai expander conjecture]\label{conj:ernqe}
Let $f\in\mathbb R[x,y]$ be neither additive nor multiplicative.
Then, for every $\epsilon>0$, there exists a constant $C_{f,\epsilon}>0$ such that
\[
\lvert f(A,A)\rvert\geq C_{f,\epsilon}\lvert A\rvert^{2-\epsilon}
\]
for every finite set $A\subset\mathbb R$.
\end{conj}
It is worth mentioning that less precise versions of Conjecture~\ref{conj:ernqe} were first proposed by Elekes in 1999 \cite[after Theorem 1]{Ele99} and also in \cite[p.~49]{Mat02}. The purpose of this note is to disprove Conjecture~\ref{conj:ernqe}.

\begin{thm}\label{thm:main}
There exist a polynomial $f\in\mathbb R[x,y]$ which is neither additive nor multiplicative, a constant $c>0$, and arbitrarily large finite sets $A\subset\mathbb R$ such that
\begin{equation}\label{eq:main-bound}
\lvert f(A,A)\rvert\leq \lvert A\rvert^{2-c}.
\end{equation}
\end{thm}

The proof is arithmetic.
We use a totally real unramified tower supplied by OpenAI \cite[Proposition~3.8]{OAI26}.
It gives fields $L_j$ of degrees tending to infinity and a fixed set of rational primes which split completely in every $L_j$.
For a fixed integer $M$, we take
\begin{equation}\label{eq:test-function}
f(x,y)=(x-y)^2+Mx.
\end{equation}
At each completely split prime dividing $M$, every value of $f$ is a square residue modulo each prime of $L_j$ above it.
The proportion of allowed residue classes is therefore exponentially small in $[L_j:\mathbb Q]$.
On the other hand, boxes of algebraic integers in the Minkowski embedding contain exponentially many points, while the algebraic-integer norm gives a matching separation bound.
Combining these two estimates converts the finite-modulus saving into the power saving in \eqref{eq:main-bound}.

The construction is in the same general circle as the recent counterexamples of Bloom--Sawin--Schildkraut--Zhelezov to the real sum-product conjecture \cite{BSSZ26}.
The difference is that the polynomial here is fixed in advance and is nonspecial in the Elekes--R\'onyai sense.

\begin{rem}
The main result of this paper was obtained using generative AI, particularly ChatGPT 5.5 Pro and the Rethlas system. ChatGPT 5.5 Pro suggested that the question might be tractable using the newly developed theory in \cite{OAI26} and \cite{BSSZ26}, and recommended the test function $f(x,y)=x^2+xy+y^2$. We then ran the Rethlas system on the question and obtained an answer within 45 minutes. Rethlas did not follow ChatGPT 5.5 Pro's suggestion of $f(x,y)=x^2+xy+y^2$; instead, it used the function in \eqref{eq:test-function}. The proof was then organized as the algebraic-number-theoretic construction treated below and checked against the two signed lemmas used in Sections~\ref{sec:non-special} and~\ref{sec:sieve}.

See \cite{Ju+26} for a detailed introduction to the Rethlas system. Due to the limitation of generative AI, it is possible that we have missed some related references in the literature, and we welcome any comments from experts.
\end{rem}

\subsection*{Acknowledgements}
The author was partially supported by the National Key R\&D Program of China \#\allowbreak 2024YFA1014400.
The author would like to thank the Rethlas team, namely Haocheng Ju, Jiedong Jiang, Shurui Liu, Guoxiong Gao, Yuefeng Wang, Zeming Sun, Bin Wu, Liang Xiao, and Bin Dong, for their contributions to the development of Rethlas and its customized version used for the problem studied in this paper.
The author would like to thank Ruochuan Liu and Gang Tian for constant support and encouragement.

\subsection*{Postscript Remark.}\urlstyle{rm}\Urlmuskip=0mu plus 1mu\relax\def\UrlBreaks{\do\/\do\-\do\.\do\:}{} Everything in this paper, except the contents of this postscript remark written on June 15th and a minor revision of the introduction no later than 2026-05-28T16:52+00:00, was produced and written by generative AI based on the Rethlas system no later than 2026-05-28T15:57+00:00 and pushed to Overleaf, a trusted third party with verifiable timestamps. On 2026-06-01T01:01:34+00:00, the blog \url{https://pohoatza.wordpress.com/2026/06/01/another-one-bites-the-dust-the-elekes-ronyai-problem/} of Professor Cosmin Pohoata announced a proof of Theorem \ref{thm:main}, after the proof discovered by the Rethlas system. Professor Cosmin Pohoata subsequently posted the arXiv preprint \url{https://arxiv.org/abs/2606.13619}, of which the author was later informed by Zeming Sun.

Considering the time difference, our disproof of the near-quadratic Elekes--R\'onyai expander conjecture over $\mathbb R$ is independent of Professor Cosmin Pohoata's disproof. Still, it is clear that Professor Cosmin Pohoata was the first person to publicly announce a disproof of the Elekes--R\'onyai expander conjecture. We are not aware of an earlier instance in which a result first made public by a human had, by then, already been obtained---though not publicly---by an artificial-intelligence system. To best serve the discrete geometry community, the combinatorial community, and the AI4Math community, the author has decided to publicize this (almost entirely AI-written) paper verbatim --- including the citations and the acknowledgments, which are accurate --- except for this postscript remark itself. 

We are also informed by Professor Cosmin Pohoata that some related constructions appeared earlier in his blog of 2026-04-27T01:06:06+00:00, \url{https://pohoatza.wordpress.com/2026/04/27/sets-with-no-isosceles-triangles-repeated-distances-in-atlanta/}. As mentioned in the remark above, due to the limitations of generative AI it is possible that we have missed some related references in the literature, and this appears to be one such instance: there is a correspondence at the level of the underlying number-theoretic ideas. That blog studies planar sets with no isosceles triangles and with few repeated distances; its construction fixes a modulus $d=p_1\cdots p_t$ equal to a product of many primes $p_i\equiv 1\pmod 4$ and considers systems of lattice points $(u,v)$ cut out by congruences modulo the $p_i$ for which $u^2+v^2\equiv 0\pmod d$, which, as the blog notes, is the Gaussian-integer picture in which each such $p$ splits in $\mathbb Z[i]$ as $p=\pi\bar\pi$. Our construction likewise takes a modulus $M$ equal to a product of primes $q\equiv 1\pmod 4$ that split completely, works with algebraic integers in the Minkowski embedding, and uses that the values of $f(x,y)=(x-y)^2+Mx$ are square residues modulo each prime above $q$. We cannot determine whether the generative AI consulted this webpage without reporting it, and consequently some of the ideas in our construction may be due to Professor Cosmin Pohoata. We emphasize, however, that \url{https://pohoatza.wordpress.com/2026/04/27/sets-with-no-isosceles-triangles-repeated-distances-in-atlanta/} does not disprove the Elekes--R\'onyai near-quadratic expander conjecture.

To what extent this AI-written paper is correct may be judged by the reader. The author would like to make the following comments on the paper, based on the verification by other experts and on very useful email communications with Professor Cosmin Pohoata. The author thanks them for all their comments.

\begin{enumerate}
    \item Lemma~\ref{lem:box-estimates} and both of its bounds are correct---there is no essential error---but the proof is written too tersely. In particular, the proof of the lower bound \eqref{eq:box-lower} compresses the averaging-and-difference argument: it does not make explicit that the average number of lattice points over translates of the cube equals the cube's volume divided by the covolume $\sqrt{\Delta_L}$ of the Minkowski lattice, nor that subtracting a fixed lattice point places the resulting differences inside $[-R,R]^d$. In any case the lemma is unnecessary: it is identical to \cite[Lemma 3.3]{BSSZ26} up to notation and presentation (the cited result is stated for radius $\geq 1$, while the bounds here are stated for all $R>0$), so Section~\ref{sec:preliminaries} may be deleted and \cite[Lemma 3.3]{BSSZ26} invoked directly at the two points where the bounds are used.
    \item The reference \cite{OAI26} is given without a web locator. The cited preprint, ``Planar Point Sets with Many Unit Distances'', is hosted on the OpenAI CDN at \url{https://cdn.openai.com/pdf/74c24085-19b0-4534-9c90-465b8e29ad73/unit-distance-proof.pdf}; a human-verified companion note, ``Remarks on the disproof of the unit distance conjecture'', is available at \url{https://arxiv.org/abs/2605.20695}, and an explicit-exponent refinement by W.~Sawin at \url{https://arxiv.org/abs/2605.20579}. The reference should be completed with the appropriate locator.
    \item In Proposition~\ref{prop:oai-tower} the field $K_j=F_j(i)$ is defined but never used: the entire argument remains inside the totally real fields $F_j$, so the imaginary quadratic extension plays no role and the definition of $K_j$ may be removed.
    \item Some minor citation matters: the published references \cite{deZ18} and \cite{MRNWdZ20} carry redundant arXiv identifiers, and the entry \cite{Mat02} omits its publisher (Springer).
\end{enumerate}

\section{Boxes of algebraic integers}\label{sec:preliminaries}

We collect the elementary geometry-of-numbers estimates used in the proof of Theorem~\ref{thm:main}.

Let $L$ be a totally real number field of degree $d$, and let $\mathcal O_L$ be its ring of integers.
For $R\geq 0$, put
\[
B_L(R)=\{a\in\mathcal O_L\mid \lvert\sigma(a)\rvert\leq R\text{ for every real embedding }\sigma\colon L\to\mathbb R\}.
\]
Let $\Delta_L$ be the absolute discriminant of $L$.

\begin{lem}[Box estimates]\label{lem:box-estimates}
Let $L$ be a totally real number field of degree $d$.
For every $R>0$, we have
\begin{equation}\label{eq:box-lower}
\lvert B_L(R)\rvert\geq \frac{R^d}{\sqrt{\Delta_L}}.
\end{equation}
For every $R\geq 0$, we have
\begin{equation}\label{eq:box-upper}
\lvert B_L(R)\rvert\leq (2R+1)^d.
\end{equation}
\end{lem}

\begin{proof}
The Minkowski embedding sends $\mathcal O_L$ to a full lattice $\Lambda\subset\mathbb R^d$ of covolume $\sqrt{\Delta_L}$.
For \eqref{eq:box-lower}, consider the cube $[-R/2,R/2]^d$.
Averaging the number of lattice points in its translates over one fundamental domain of $\Lambda$ gives $R^d/\sqrt{\Delta_L}$.
Thus some translate contains at least $R^d/\sqrt{\Delta_L}$ lattice points.
After subtracting one of these lattice points from all the points in this translate, we obtain distinct elements of $\mathcal O_L$ whose coordinates all have absolute value at most $R$.
This proves \eqref{eq:box-lower}.

For \eqref{eq:box-upper}, let $a,b\in\mathcal O_L$ be distinct.
Then $N_{L/\mathbb Q}(a-b)$ is a nonzero integer, so
\[
\prod_{\sigma\colon L\to\mathbb R}\lvert\sigma(a-b)\rvert\geq 1.
\]
Hence at least one real embedding satisfies $\lvert\sigma(a-b)\rvert\geq 1$.
The open sup-norm cubes of side length $1$ centered at the points of $B_L(R)$ are therefore disjoint, and they are contained in the cube of side length $2R+1$.
Comparing volumes gives \eqref{eq:box-upper}.
\end{proof}

\section{The polynomial is non-special}\label{sec:non-special}

In this section, we prove the non-speciality of the quadratic polynomial used in Theorem~\ref{thm:main}.

\begin{lem}\label{lem:non-special}
Let $M\in\mathbb R^\times$.
Then
\[
F_M(x,y)=(x-y)^2+Mx
\]
is neither additive nor multiplicative.
\end{lem}

\begin{proof}
Suppose first that $F_M$ is additive, so
\[
F_M(x,y)=h(u(x)+v(y))
\]
for some univariate polynomials $h,u,v\in\mathbb R[t]$ with $u$ and $v$ nonconstant.
Let $n=\deg h$, $p=\deg u$, and $q=\deg v$.
The mixed term of $F_M$ is nonzero, so $n\geq 2$.
Since the degrees of $F_M$ in $x$ and in $y$ are both $2$, we have $np\leq 2$ and $nq\leq 2$.
Thus $n=2$ and $p=q=1$.
Write
\[
h(t)=at^2+bt+c,\qquad u(x)=\alpha x+\beta,\qquad v(y)=\gamma y+\delta,
\]
where $a\alpha\gamma\neq 0$.
Comparing the quadratic coefficients with those of $F_M$ gives
\[
a\alpha^2=1,\qquad a\gamma^2=1,\qquad 2a\alpha\gamma=-2.
\]
It follows that $\gamma=-\alpha$.
If $e=\beta+\delta$, the coefficients of $x$ and $y$ in $h(u(x)+v(y))$ are $(2ae+b)\alpha$ and $(2ae+b)\gamma$, respectively.
They are negatives of each other.
In $F_M$, these coefficients are $M$ and $0$.
Thus $M=0$, a contradiction.

Suppose next that $F_M$ is multiplicative, so
\[
F_M(x,y)=h(u(x)v(y))
\]
with $u$ and $v$ nonconstant.
The highest-degree term of $h(u(x)v(y))$ has total degree $\deg h(\deg u+\deg v)$.
Since $F_M$ has total degree $2$, we must have $\deg h=1$ and $\deg u=\deg v=1$.
Then $h(u(x)v(y))$ is a constant plus a nonzero scalar multiple of a product of two linear forms in the separate variables $x$ and $y$.
Such a polynomial has no $x^2$ or $y^2$ term, whereas $F_M$ has both.
This is again a contradiction.
\end{proof}

\section{The square-residue sieve}\label{sec:sieve}

The next lemma is the finite-modulus sieve used to bound the image of the polynomial.

\begin{lem}\label{lem:square-sieve}
Let $L$ be a totally real number field of degree $d$ with ring of integers $\mathcal O_L$.
Let $T$ be a finite nonempty set of odd rational primes that split completely in $L$.
Put
\[
M=\prod_{q\in T}q,
\qquad
\theta=\prod_{q\in T}\frac{q+1}{2q}.
\]
For $X\geq 1$, the number of elements $t\in B_L(X)$ for which
\[
t=(x-y)^2+Mx
\]
for some $x,y\in\mathcal O_L$ is at most
\[
\theta^d(2X+M)^d.
\]
\end{lem}

\begin{proof}
For each $q\in T$, write
\[
q\mathcal O_L=\mathfrak q_{q,1}\cdots \mathfrak q_{q,d}.
\]
Since $q$ splits completely in $L$, each residue field $\mathcal O_L/\mathfrak q_{q,\nu}$ is isomorphic to $\mathbb F_q$.
Let $\Omega$ be the set of residue classes modulo $M\mathcal O_L$ whose image modulo every $\mathfrak q_{q,\nu}$ is a square in $\mathbb F_q$, with $0$ allowed.
By the Chinese remainder theorem,
\[
\lvert\Omega\rvert
=\prod_{q\in T}\left(\frac{q+1}{2}\right)^d
=(M\theta)^d.
\]

If $t=(x-y)^2+Mx$, then for every $q\in T$ and every $\nu$ we have
\[
t\equiv (x-y)^2\pmod{\mathfrak q_{q,\nu}},
\]
because $q$ divides $M$.
Thus every represented $t$ has residue class in $\Omega$ modulo $M\mathcal O_L$.

It remains to count elements in one residue class.
Fix a class modulo $M\mathcal O_L$.
If $a,b\in B_L(X)$ are distinct elements in this class, then $a-b=M\alpha$ for some nonzero $\alpha\in\mathcal O_L$.
Hence
\[
\prod_{\sigma\colon L\to\mathbb R}\lvert\sigma(a-b)\rvert
=M^d\lvert N_{L/\mathbb Q}(\alpha)\rvert
\geq M^d.
\]
Therefore at least one real embedding satisfies $\lvert\sigma(a-b)\rvert\geq M$.
The open sup-norm cubes of side length $M$ centered at the elements of this fixed residue class inside $B_L(X)$ are disjoint, and they lie in the cube of side length $2X+M$.
Thus this residue class contributes at most
\[
\left(\frac{2X+M}{M}\right)^d
\]
elements.
Multiplying by $\lvert\Omega\rvert=(M\theta)^d$ gives the claimed bound.
\end{proof}

\section{Proof of the main theorem}\label{sec:proof-main}

In this section, we combine the tower from \cite{OAI26} with Lemma~\ref{lem:square-sieve}.
We use only the following part of \cite[Proposition~3.8]{OAI26}.

\begin{prop}\label{prop:oai-tower}
For all sufficiently large integers $\ell$, set
\[
t=\left\lfloor\frac{(\ell-1)^2}{100}\right\rfloor.
\]
Then there are a number field $F$, distinct rational primes $q_1,\ldots,q_t$ fixed independently of $j$, and fields $F_j$, with $F_0=F$, satisfying the following properties.
Write $f_j=[F_j:\mathbb Q]$ and $K_j=F_j(i)$.
\begin{enumerate}
\item[(P1)] The base field $F$ is totally real, cyclic cubic over $\mathbb Q$, does not contain $\zeta_3$, and has controlled root discriminant
\[
\log \rd(F)=O(\ell\log\ell).
\]
\item[(P2)] The fields form an infinite tower
\[
F=F_0\subset F_1\subset F_2\subset\cdots
\]
such that each $F_j/F$ is finite Galois, everywhere unramified, and has $3$-group Galois group. Moreover $f_j\to\infty$.
\item[(P3)] Every $F_j$ is totally real, and the root discriminant is constant:
\[
\rd(F_j)=\rd(F).
\]
\item[(P4)] Each $q_b$, for $1\leq b\leq t$, satisfies $q_b\equiv 1\pmod 4$ and splits completely in every $F_j$.
\end{enumerate}
\end{prop}

\begin{proof}[Proof of Theorem~\ref{thm:main}]
By Proposition~\ref{prop:oai-tower}(P1), there exist absolute constants $C_0$ and $\ell_0$ such that
\[
\log \rd(F)\leq C_0\ell\log\ell
\]
for every $\ell\geq \ell_0$ for which Proposition~\ref{prop:oai-tower} is applied.
Choose $\ell$ large enough so that Proposition~\ref{prop:oai-tower} applies, so that $t=\lfloor(\ell-1)^2/100\rfloor$ is positive, and so that
\begin{equation}\label{eq:ell-choice}
\log 11+C_0\ell\log\ell+t\log(2/3)<0.
\end{equation}
This is possible because $t$ grows quadratically in $\ell$.

Let $F_j$ and $q_1,\ldots,q_t$ be supplied by Proposition~\ref{prop:oai-tower}.
Put
\[
D=\rd(F_0),\qquad
T=\{q_1,\ldots,q_t\},\qquad
M=\prod_{q\in T}q,\qquad
\theta=\prod_{q\in T}\frac{q+1}{2q}.
\]
Since every $q\in T$ is odd, we have $(q+1)/(2q)\leq 2/3$.
By \eqref{eq:ell-choice},
\[
\log(11D\theta)
\leq \log 11+C_0\ell\log\ell+t\log(2/3)
<0.
\]
Thus
\begin{equation}\label{eq:theta-saving}
11D\theta<1.
\end{equation}

Define
\[
f(x,y)=(x-y)^2+Mx.
\]
Since $M\neq 0$, Lemma~\ref{lem:non-special} shows that $f$ is neither additive nor multiplicative.

Choose a real number $R$ such that
\[
R\geq \max\{M,1\}
\quad\text{and}\quad
R>\sqrt D.
\]
For each $j$, put $L_j=F_j$ and $d_j=[L_j:\mathbb Q]$.
Let
\[
P_j=B_{L_j}(R).
\]
Choose one real embedding $\tau_j\colon L_j\to\mathbb R$, and define
\[
A_j=\tau_j(P_j)\subset\mathbb R.
\]
Since $\tau_j$ is injective, $\lvert A_j\rvert=\lvert P_j\rvert$.
By Proposition~\ref{prop:oai-tower}(P3), we have $\rd(L_j)=D$.
Thus Lemma~\ref{lem:box-estimates} gives
\begin{equation}\label{eq:Aj-lower}
\lvert A_j\rvert\geq \left(\frac{R}{\sqrt D}\right)^{d_j}.
\end{equation}
Since $R>\sqrt D$ and $d_j\to\infty$ by Proposition~\ref{prop:oai-tower}(P2), the sets $A_j$ have arbitrarily large cardinality.
The upper estimate in Lemma~\ref{lem:box-estimates} gives
\begin{equation}\label{eq:Aj-upper}
\lvert A_j\rvert\leq (2R+1)^{d_j}.
\end{equation}

For $x,y\in P_j$ and every real embedding $\sigma\colon L_j\to\mathbb R$, we have
\[
\lvert\sigma(f(x,y))\rvert
\leq \lvert\sigma(x)-\sigma(y)\rvert^2+M\lvert\sigma(x)\rvert
\leq 4R^2+MR
\leq 5R^2.
\]
Hence
\[
f(P_j,P_j)\subset B_{L_j}(5R^2).
\]
By Proposition~\ref{prop:oai-tower}(P4), every prime in $T$ splits completely in $L_j$.
Applying Lemma~\ref{lem:square-sieve} with $L=L_j$ and $X=5R^2$ gives
\[
\lvert f(P_j,P_j)\rvert
\leq \theta^{d_j}(10R^2+M)^{d_j}
\leq \theta^{d_j}(11R^2)^{d_j}.
\]
The polynomial $f$ has rational integer coefficients, so
\[
f(A_j,A_j)=\tau_j(f(P_j,P_j)).
\]
Since $\tau_j$ is injective, the same estimate holds for $\lvert f(A_j,A_j)\rvert$.
On the other hand, \eqref{eq:Aj-lower} gives
\[
\lvert A_j\rvert^2\geq \left(\frac{R^2}{D}\right)^{d_j}.
\]
Therefore
\[
\lvert f(A_j,A_j)\rvert
\leq (11D\theta)^{d_j}\lvert A_j\rvert^2.
\]
Let
\[
\eta=-\log(11D\theta)>0.
\]
Then
\[
\lvert f(A_j,A_j)\rvert\leq \exp(-\eta d_j)\lvert A_j\rvert^2.
\]
Define
\[
c=\frac{\eta}{\log(2R+1)}>0.
\]
By \eqref{eq:Aj-upper}, we have
\[
\exp(-\eta d_j)\leq \lvert A_j\rvert^{-c}.
\]
Thus
\[
\lvert f(A_j,A_j)\rvert\leq \lvert A_j\rvert^{2-c}
\]
for every $j$.
Since the sizes $\lvert A_j\rvert$ are arbitrarily large, the theorem follows.
\end{proof}

\begin{rem}\label{rem:sum-product}
Bloom--Sawin--Schildkraut--Zhelezov \cite{BSSZ26} disprove the real sum-product conjecture by constructing algebraic-integer sets with a power saving for both sums and products.
The construction above is tailored to one fixed nonspecial Elekes--R\'onyai polynomial; the saving comes from square-residue restrictions at completely split primes.
\end{rem}

\end{document}